\documentclass[10pt,leqno]{amsart}
\usepackage{indentfirst, amssymb,amsmath,amsthm, enumerate}
\newtheorem*{theoA}{Theorem A}

\newtheorem{theo}{Theorem}[section]
\newtheorem{lem}{Lemma}[section]

\newtheorem{defi}{Definition}[section]

\newcommand{\ol}{\overline}
\newcommand{\be}{\begin{equation}}
\newcommand{\ee}{\end{equation}}
\newcommand{\beas}{\begin{eqnarray*}}
\newcommand{\eeas}{\end{eqnarray*}}
\newcommand{\bea}{\begin{eqnarray}}
\newcommand{\eea}{\end{eqnarray}}
\newcommand{\lra}{\longrightarrow}

\numberwithin{equation}{section}
\renewcommand{\vline}{\mid}
\numberwithin {equation}{section}
\numberwithin {lem}{section}
\numberwithin {theo}{section}
\numberwithin {defi}{section}
\numberwithin {rem}{section}
\AtBeginDocument{{\noindent\small Afr. Mat. (2015) 26:1561–1572\\
DOI: 10.1007/s13370-014-0305-4}}
\begin{document}
\title[A new type of unique range set with deficient values] {A new type of unique range set with deficient values}
\date{}
\author[A. Banerjee and B. Chakraborty ]{ Abhijit Banerjee * and Bikash Chakraborty. }
\address{Department of Mathematics, University of Kalyani, West Bengal 741235, India.}
\email{abanerjee\_kal@yahoo.co.in, abanerjee\_kal@rediffmail.com}
\email{bikashchakraborty.math@yahoo.com, bikashchakrabortyy@gmail.com.}
\let\thefootnote\relax
\thanks{Type set by \AmS -\LaTeX}
\thanks{$^{*}$ The first author is thankful to DST-PURSE programme for financial assistance.}
\begin{abstract} In the paper, we introduce a new type of unique range set for meromorphic function having deficient values which will improve all the previous result in this aspect. \end{abstract}
\maketitle
\section{Introduction Definitions and Results}
In this paper by meromorphic functions we will always mean meromorphic functions in the complex plane. It will be convenient to let $E$ denote any set of positive real numbers of finite linear measure, not necessarily the same at each occurrence. For any non-constant meromorphic function $h(z)$ we denote by  $S(r,h)$ any quantity satisfying
$$S(r,h)=o(T(r,h))\;\;\;\;\;( r\lra \infty, r\not\in E).$$

We denote by $T(r)$ the maximum of $T(r,f)$ and $T(r,g)$. The notation $S(r)$ denotes any quantity satisfying $S(r)=o(T(r))$ as $r\lra \infty$, $r\not\in E$.
Throughout this paper, we denote \beas\Theta(a,f) = 1- \displaystyle\limsup_{r\longrightarrow \infty }\frac{\ol N(r,a;f)}{T(r,f)},\eeas where $a$ is a value in the extended complex plane. \par
We adopt the standard notations of the Nevanlinna theory of meromorphic functions as explained in \cite{6}.
%For $a\in\mathbb{C}\cup\{\infty\}$, we define $$\Theta(a;f)= 1- \limsup\limits_{r\lra \infty}\frac{\ol N(r,a;f)}{T(r,f)}.$$

Let $f$ and $g$ be two non-constant meromorphic functions and let $a$ be a finite complex number. We say that $f$ and $g$ share the value $a$ CM (counting multiplicities), provided that $f-a$ and $g-a$ have the same zeros with the same multiplicities. Similarly, we say that $f$ and $g$ share the value $a$-IM (ignoring multiplicities), provided that $f-a$ and $g-a$ have the same set of zeros, where the multiplicities are not taken into account.
In addition we say that $f$ and $g$ share $\infty$ CM (IM), if $1/f$ and $1/g$ share $0$ CM (IM).

Let $S$ be a set of distinct elements of $\mathbb{C}\cup\{\infty\}$ and $E_{f}(S)=\bigcup_{a\in S}\{z: f(z)=a\}$, where each zero is counted according to its multiplicity. If we do not count the multiplicity, then the set $ \bigcup_{a\in S}\{z: f(z)=a\}$ is denoted by $\ol E_{f}(S)$.
If $E_{f}(S)=E_{g}(S)$ we say that $f$ and $g$ share the set $S$ CM. On the other hand, if $\ol E_{f}(S)=\ol E_{g}(S)$, we say that $f$ and $g$ share the set $S$ IM. Evidently, if $S$ contains only one element, then it coincides with the usual definition of CM (respectively, IM) sharing of values.
%In 1926, R. Nevanlinna showed that a meromorphic function on the complex plane $\mathbb{C}$ is uniquely determined by the pre-images, ignoring multiplicities, of 5 distinct values. A few years later, he showed that when multiplicities are considered, 4 points
%are sufficient and in this case either the functions coincide or one is the bilinear transformation of the other.
%In 1926, R. Nevanlinna proved two fundamental results on shared values. His famous five value theorem gives an upper bound on the number of distinct values two different meromorphic functions can share IM.
%Taking multiplicities into account Nevanlinna proved that if two meromorphic functions share four distinct values CM then either they coincide or one is the fractional linear transformation of the other.
%This two theories can be considered as the gateway to uniqueness theory which has been flourished immensely for the last two decades.

%In \cite{4} Gross extended the study of uniqueness by determining an entire function with the help of single pre-image of a finite set $S$ counting multiplicities.
%In 1982 F. Gross and C. C. Yang \cite{5} proved the following theorem:
%\begin{theoA}\cite{5} Let $S=\{z\in \mathbb{C} :e^{z}+z=0\}$. If two entire functions $f$, $g$ satisfy  $E_{f}(S)=E_{g}(S)$ then $f\equiv g$.\end{theoA}
In continuation with the famous question proposed in \cite{4}, in 1982, F.Gross and C.C.Yang \cite{5} first introduced the novel idea of unique range set for entire function. The analogous definition for meromorphic functions can also be given in similar fashion. Below we are recalling the same. \par
Let a set $S\subset \mathbb{C} $ and $f$ and $g$ be two non-constant meromorphic (entire) functions. If $E_{f}(S)=E_{g}(S)$ implies $f\equiv g$ then $S$ is called a unique range set for meromorphic (entire) functions or in brief URSM (URSE). \par
In 1997, Yi \cite{15a} introduced the analogous definition for reduced unique range sets. We shall call any set $S\subset \mathbb{C} $ a unique range set for meromorphic (entire) functions ignoring multiplicity (URSM-IM) (URSE-IM) or a reduced unique range set for meromorphic (entire) functions (RURSM) (RURSE) if $\ol E_{f}(S)=\ol E_{g}(S)$ implies $f\equiv g$ for any pair of non-constant meromorphic (entire) functions.

%It is to be observed that since the range set $S$ given in {\it Theorem A} is an infinite set, {\it Theorem A} can not be considered as an exact solution solution to the problem of Gross.\par
During the last few years the notion of unique as well as reduced unique range sets have been generating an increasing interest among the researchers and naturally a new area of research have been developed under the aegis of uniqueness theory.
The prime concern of the researchers is to find new unique range sets or to make the cardinalities of the existing range sets as small as possible imposing sum restrictions on the deficiencies of the generating meromorphic functions.
To see the remarkable progress in this regard readers can make a glance to \cite{1}, \cite{1c}-\cite{1e}, \cite{3}, \cite{9a}-\cite{9b}, \cite{12}, \cite{14}-\cite{15}.
\par In 1994, H.X.Yi \cite{14} exhibited a URSE with $15$ elements and in 1995 P.Li and C.C.Yang \cite{12} exhibited a URSM with $15$ elements and a URSE with $7$ elements.
Till date the URSM with $11$ elements and R-URSM with $17$ elements are the smallest available URSM and R-URSM obtained by Frank-Reinders \cite{3} and Bartels \cite{1e} respectively. This URSM by Frank of Reinders is highlighted by a number of researchers.  \par
In 1995, Li and Yang \cite{12} first elucidated the fact that the finite URSM's are nothing but the set of distinct zeros of some suitable polynomials and subsequently the study of the characteristics of these underlying polynomials should also be given utmost priority.

Li and Yang \cite{12}, called a polynomial $P$ in $\mathbb{C}$, as uniqueness polynomial for meromorphic (entire) functions, if for any two non-constant meromorphic (entire) functions $f$ and $g$, $P(f)\equiv P(g)$ implies $f\equiv g$.
We say $P$ is a UPM (UPE) in brief. \par
On the other hand, T. T. H. An, J. T. Wang and P. Wong \cite{1a} called a polynomial $P$ in $\mathbb{C}$ as strong uniqueness polynomial
for meromorphic (entire) functions if for any non-constant meromorphic (entire) functions $f$ and $g$, $P(f)\equiv cP(g)$ implies $f\equiv g$, where $c$ is a suitable nonzero constant. In this case we say $P$ is SUPM (SUPE) in brief.\par

%Suppose that $P$ is a polynomial of degree $n$ in $\mathbb{C}$ having only simple zeros and $S$ be the set of all zeros of $P$. If $S$ is a URSM (URSE), then from the definition it follows that $P$ is UPM (UPE). However the converse is not, in general, true.
%For, $P(z)=az+b$ ($a\not =0$) is clearly a UPM but for $\display f=-\frac{b}{a}e^{z}$ and $\display g=-\frac{b}{a}e^{-z}$ we see that
%$E_{f}(S)=E_{g}(S)$, where $S=\{-\frac{b}{a}\}$ is the set of zeros of $P(z)=az+b$. \par
In 2000, H. Fujimoto \cite{2} first discovered a special property of a polynomial, which was recently termed as critical injection property in \cite{1b}.
Critical injection property of a polynomial may be stated as follows : A polynomial $P$ is said to satisfy critical injection property if $P(\alpha )\not =P(\beta )$ for any two distinct zeros $\alpha $, $\beta $ of the derivative $P'$.\par
Clearly the inner meaning of critical injection property is that the polynomial $P$ is injective on the set of distinct zeros of $P^{'}$, which are known as critical points of $P$. Naturally a polynomial with this property may be called a critically injective polynomial.
Let $P(z)$ be a monic polynomial without multiple zero whose derivatives has mutually distinct $k$ zeros given by $d_{1}, d_{2}, \ldots, d_{k}$ with multiplicities $q_{1}, q_{2}, \ldots, q_{k}$ respectively. The following theorem of Fujimoto helps us to find many uniqueness polynomials.
\par
\begin{theoA}\cite{2a} Suppose that $P(z)$ is critically injective.  Then $P(z)$ will be a uniqueness polynomial if and only if $$\sum \limits_{1\leq l<m\leq k}q_{_{l}}q_{m}>\sum \limits_{l=1}^{k} q_{_{l}}.$$
In particular the above inequality is always satisfied whenever $k\geq 4$. When $k=3$ and $\max \{q_{1},q_{2},q_{3}\}\geq 2$ or when $k=2$, $\min \{q_{1},q_{2}\}\geq 2$ and $q_{1}+q_{2}\geq 5$ then also the above inequality holds.\end{theoA}
For $k=1$, taking $P(z)=(z-a)^{q}-b$ for some constants $a$ and $b$ with $b\not=0$ and an integer $q\geq 2$, it is easy to verify that for an arbitrary non-constant meromorphic function $g$ and a constant $c(\not=1)$ with $c^{q}=1$, the function $g:=cf+(1-c)a(\not=f)$ satisfies the condition $P(f)=P(g)$.\par
A recent development in the uniqueness theory has been to consider the notion of weighted sharing of values and sets \{\cite{8}, \cite{9}\} which is a scaling between CM sharing and IM sharing and measures a gradual
increment from IM sharing to CM sharing.
\par
Let $k$ be a non-negative integer or infinity. For $a\in\mathbb{C}\cup\{\infty\}$ we denote by
$E_{k}(a;f)$ the set of all $a$-points of $f$, where an $a$-point of multiplicity $m$ is counted $m$
times if $m\leq k$ and $k+1$ times if $m>k$. \par
If for two meromorphic functions $f$ and $g$ we have $E_{k}(a;f)=E_{k}(a;g)$, then we say that $f$
and $g$ share the value $a$ with weight $k$.\par

The IM and CM sharing respectively correspond to weight $0$ and $\infty$.\par
For $S\subset \mathbb{C}\cup\{\infty\}$ we define $E_{f}(S,k)$ as $$E_{f}(S,k)=\displaystyle\bigcup_{a\in S}E_{k}(a;f),$$
where $k$ is a nonnegative integer or infinity. Clearly $E_{f}(S)=E_{f}(S,\infty)$.\par
The main intention of the paper is to introduce a new type of unique range set for meromorphic function which improve all the previous results in this aspect specially those of \cite{1c} and \cite{1d} by removing the ``max" conditions in deficiencies.
Henceforth for two positive integers $n$, $m$ we shall denote by $P(z)$ the following polynomial.
$$P(z)=\displaystyle\sum_{i=0}^{m}\hspace{.05in}\binom{m}{i}\frac{(-1)^{i}}{n+m+1-i} z^{n+m+1-i}+1=Q(z)+1,$$ where $Q(z)=\displaystyle\sum_{i=0}^{m}\hspace{.05in}\binom{m}{i}\frac{(-1)^{i}}{n+m+1-i} z^{n+m+1-i}$.
Clearly $P^{'}(z)=z^{n}(z-1)^{m}$. So $P(0)=1$ and $P(1)=Q(1)+1$.\\ Following theorem is the main result of the paper.
\begin{theo}\label{t1.1} Let $n(\geq 3)$, $m(\geq 3)$ be two positive integers. Suppose that $S=\{z: P(z)=0\}$. Let $f$ and $g$ be two non-constant meromorphic functions such that $E_{f}(S,l)=E_{g}(S,l)$.
Now if one of the following conditions holds:
\begin{enumerate}
\item[(a)] $l\geq 2$  and $\Theta _{f}+\Theta _{g}>(9-(n+m))$
\item[(b)] $l=1$ and  $\Theta _{f}+\Theta _{g}>(10-(n+m))$;
\item[(c)] $l=0$ and  $\Theta _{f}+\Theta _{g}>(15-(n+m))$;
\end{enumerate}
then $f\equiv g$, where $\Theta _{f}=2\Theta (0,f)+2\Theta (\infty,f)+\Theta (1,f)+\frac{1}{2}\min\{\Theta (1;f),\Theta (1;g)\}$ and $\Theta _{g}$ is similarly defined.  \end{theo}
We now explain some definitions and notations which are used in the paper.
\begin{defi} \label{d3}\cite{7} For $a\in\mathbb{C}\cup\{\infty\}$ we denote by $N(r,a;f\mid=1)$ the counting function of simple $a$-points of $f$. For a positive integer $m$ we denote by $N(r,a;f\vline\leq m)\; (N(r,a;f\vline\geq m)$ by the counting function of those $a$-points of $f$ whose multiplicities are not greater(less) than $m$ where each $a$-point is counted according to its multiplicity.

$\ol N(r,a;f\mid\leq m)$ and $\ol N(r,a;f\mid\geq m)$ are the reduced counting function of $N(r,a;f\mid\leq m)$ and $N(r,a;f\mid\geq m)$ respectively.

Also $N(r,a;f\mid <m), N(r,a;f\mid >m), \ol N(r,a;f\mid <m)\; and \;\ol N(r,a;f\mid >m)$ are defined analogously.  \end{defi}

\begin{defi}\label{d5}\cite{16} Let $f$ and $g$ be two non-constant meromorphic functions such that $f$ and $g$ share $(a,0)$. Let $z_{0}$ be an $a$-point of $f$ with multiplicity $p$, an $a$-point of $g$ with multiplicity $q$. We denote by $\ol N_{L}(r,a;f)$ the reduced counting function of those $a$-points of $f$ and $g$ where $p>q$, by $N^{1)}_{E}(r,a;f)$ the counting function of those $a$-points of $f$ and $g$ where $p=q=1$, by $\ol N^{(2}_{E}(r,a;f)$ the reduced counting function of those $a$-points of $f$ and $g$ where $p=q\geq 2$. In the same way we can define $\ol N_{L}(r,a;g),\;  N^{1)}_{E}(r,a;g),\;  \ol N^{(2}_{E}(r,a;g).$ In a similar manner we can define $\ol N_{L}(r,a;f)$ and $\ol N_{L}(r,a;g)$ for $a\in\mathbb{C}\cup\{\infty\}$. \end{defi}
When $f$ and $g$ share $(a,m)$, $m\geq 1$ then $N^{1)}_{E}(r,a;f)=N(r,a;f\mid=1)$.
\begin{defi}\label{d6} We denote by $\ol N(r,a;f\mid=k)$ the reduced counting function of those $a$-points of $f$ whose multiplicities is exactly $k$, where $k\geq 2$ is an integer. \end{defi}
\begin{defi} \label{d7}\cite{8, 9} Let $f$, $g$ share a value $a$ IM. We denote by $\ol N_{*}(r,a;f,g)$ the reduced counting function of those $a$-points of $f$ whose multiplicities differ from the multiplicities of the corresponding $a$-points of $g$.

Clearly $\ol N_{*}(r,a;f,g) \equiv \ol N_{*}(r,a;g,f)$ and $\ol N_{*}(r,a;f,g)=\ol N_{L}(r,a;f)+\ol N_{L}(r,a;g)$\end{defi}

\section{Lemmas} In this section we present some lemmas which will be needed in the sequel.
%We suppose that the derivative $P'$ has distinct zeros $\beta _{1}, \beta _{2},\ldots, \beta _{k}$ with respective multiplicities $q_{1}, q_{2},\ldots, q_{k}$. So $P^{'}(z)=(z-\beta _{1})^{q_{1}}(z-\beta _{2})^{q_{2}}\ldots(z-\beta _{k})^{q_{k}}$ where $q_{1}+q_{2}+\ldots+q_{k}=n-1$ .\par
Let, unless otherwise stated $F$ and $G$ be two non-constant meromorphic functions given by $F=P(f)$ and $G=P(g)$.
Henceforth we shall denote by $H$ the following function \be{\label{e2.1}}H=\left(\frac{F^{''}}{F^{'}}-\frac{2F^{'}}{F}\right)-\left(\frac{G^{''}}{G^{'}}-\frac{2G^{'}}{G}\right).\ee
\begin{lem}\label{l1}\cite{11} Let $f$ be a non-constant meromorphic function and let \[R(f)=\frac{\sum\limits _{k=0}^{n} a_{k}f^{k}}{\sum \limits_{j=0}^{m} b_{j}f^{j}}\] be an irreducible rational function in $f$ with constant coefficients $\{a_{k}\}$ and $\{b_{j}\}$where $a_{n}\not=0$ and $b_{m}\not=0$ Then $$T(r,R(f))=dT(r,f)+S(r,f),$$ where $d=\max\{n,m\}$.\end{lem}
\begin{lem}\label{l2} If $F$, $G$ are two non-constant meromorphic functions such that they share $(0,0)$ and $H\not\equiv 0$ then $$N^{1)}_{E}(r,0;F\mid=1)=N^{1)}_{E}(r,0;G\mid=1)\leq N(r,H)+S(r,f)+S(r,g).$$\end{lem}
\begin{proof}By the Lemma of Logarithmic derivative we obtain $$m(r,H)=S(r,f)+S(r,g)(:=S(r)).$$
\par By Laurent expansion of $H$ we can easily verify that each simple zero of $F$ (and so of $G$) is a zero of $H$. Hence
\beas N^{1)}_{E}(r,0;F\mid=1)&=&N^{1)}_{E}(r,0;G\mid=1)\nonumber\\&\leq& N(r,0;H)\nonumber\\&\leq& T(r,H)+O(1)\nonumber\\&=& N(r,\infty;H)+S(r,f)+S(r,g).\nonumber\eeas\end{proof}
\begin{lem}\label{l3} Let $S$ be the set of zeros of $P$. If for two non-constant meromorphic functions $f$ and $g$, $E_{f}(S,0)=E_{g}(S,0)$ and $H\not\equiv 0$ then \beas \ol N(r,\infty;H)&\leq& \ol N(r,0;f)+\ol N(r,1;f)+\ol N(r,0;g)+\ol N(r,1;g)+\ol N_{*}(r,0;F,G)\\& &+\ol N(r,\infty;f)+\ol N(r,\infty;g)+\ol N_{0}(r,0;f^{'})+\ol N_{0}(r,0;g^{'}),\eeas where by $\ol N_{0}(r,0;f^{'})$ we mean the reduced counting function of those zeros of $f^{'}$ which are not the zeros of $Ff(f-1)$ and $\ol N_{0}(r,0;g^{'})$ is similarly defined.\end{lem}
\begin{proof} Since $E_{f}(S,0)=E_{g}(S,0)$ it follows that $F$ and $G$ share $(0,0)$. Also we observe that $F^{'}=f^{n}(f-1)^{m}f^{'}$.
It can be easily verified that possible poles of $H$ occur at (i) poles of $f$ and $g$, (ii) those $0$-points of $F$ and $G$ with different multiplicities, (iii) zeros of $f^{'}$ which are not the zeros of $Ff(f-1)$, (iv) zeros of $g^{'}$ which are not zeros of $Gg(g-1)$, (v) $0$ and $1$ points of $f$ and $g$.
Since $H$ has only simple poles, the lemma follows from above. This proves the lemma.\end{proof}
%\begin{lem}\label{l4}\cite{2} Under the same situation as in {\it Theorem B} we assume further more that there are two meromorphic functions $f$ and $g$ such that $$\frac{1}{P(f)}\equiv \frac{c_{0}}{P(g)}+c_{1}$$ for any two constants $c_{0}(\not=0)$ and $c_{1}$. If $n \geq 5$  then $c_{1}=0$.  \end{lem}
\begin{lem}\label{l4} $Q(1)$ is not an integer. In particular, $P(1)\not =-1$, where $n\geq 3$, $m\geq 3$ are integers.\end{lem}
\begin{proof}We claim that \beas S_{n}(m)&=& \sum\limits_{i=0}^{m}\binom{m}{i} \frac{(-1)^i}{n+m+1-i}\\&=& \frac{\binom{m}{0}}{n+m+1} - \frac{\binom{m}{1}}{n+m+1-1} + \ldots+ (-1)^{m}\frac{\binom{m}{m}}{n+1}\\&=&\frac{(-1)^m m!}{(n+m+1)(n+m)\ldots(n+1)}.\eeas
We prove the claim by method of induction on $m$. \par
At first for $m=3$ we get \beas S_{n}(3)&=&\frac{1}{n+4}-\frac{3}{n+3}+\frac{3}{n+2}-\frac{1}{n+1}\\&=&\frac{(-1)^3\cdot 3!}{(n+4)(n+3)(n+2)(n+1)}.\eeas
So, $S_{n}(m)$ is true for $m=3$. Now we assume that $S_{n}(m)$ is true for $m=k$, where $k$ is any given positive integer such that $k\geq3$. Now we will show that $S_{n}(m)$ is true for $m=k+1$.
i.e., \beas {\frac{\binom{k+1}{0}}{n+k+2}} - {\frac{\binom{k+1}{1}}{n+k+1}} +\ldots + (-1)^{k+1}\frac{\binom{k+1}{k+1}}{n+1}=\frac{(-1)^{(k+1)} {(k+1)!}}{(n+k+2)(n+k+1)\ldots(n+1)}.\eeas
Using induction hypothesis we have \beas S_{n}(k+1)&=&\frac{1}{n+k+2} - \frac{k+1}{n+k+1} + \frac{(k+1)k}{2(n+k)}- \ldots + \frac{(-1)^{k+1}}{n+1}\\
&=&\left[\frac{1}{n+k+2} - \frac{k}{n+k+1} + \frac{k(k-1)}{2(n+k)} - \ldots+ \frac{(-1)^{k}}{n+2}\right] \\&&-\left[\frac{1}{n+k+1} - \frac{2k}{2(n+k)} + \frac{3k(k-1)}{2.3(n+k-1)} - \ldots+ \frac{(-1)^{k}}{n+1}\right]\\
&=& \left[\frac{\binom{k}{0}}{(n+1)+k+1} - \frac{\binom{k}{1}}{(n+1)+k} + \frac{\binom{k}{2}}{(n+1)+k-1} - \ldots+ (-1)^{k}\frac{\binom{k}{k}}{(n+1)+1}\right]
\\&&-\left[\frac{\binom{k}{0}}{n+k+1} - \frac{\binom{k}{1}}{(n+k)} + \frac{\binom{k}{2}}{(n+k-1)} - \ldots + (-1)^{k}\frac{\binom{k}{k}}{n+1}\right]\\&=& S_{n+1}(k)-S_{n}(k)\\&=&\frac{(-1)^{k} k!}{(n+k+2)(n+k+1)\ldots(n+2)}-\frac{(-1)^{k} k!}{(n+k+1)(n+k)\ldots(n+1)}\\&=& \frac{(-1)^{(k+1)} (k+1)!}{(n+k+2)(n+k+1)\ldots(n+1)}.\eeas
So our claim has been established. We note that $S_{n}(m)=(-1)^{m}\prod\limits_{i=1}^{m}\frac{i}{(n+i)}\frac{1}{(n+m+1)}$ and hence it can not be an integer.
In particular we have proved that $Q(1)\not =-2$ i.e., $P(1)\not=-1$. \end{proof}
\begin{lem}\label{l6}\cite{10} If $N(r,0;f^{(k)}\vline f\not=0)$ denotes the counting function of those zeros of $f^{(k)}$ which are not the zeros of $f$, where a zero of $f^{(k)}$ is counted according to its multiplicity then $$N(r,0;f^{(k)}\vline f\not=0)\leq k\ol N(r,\infty;f)+N(r,0;f\vline <k)+k\ol N(r,0;f\vline\geq k)+S(r,f).$$\end{lem}
\section {Proof of the theorem}
\begin{proof} [Proof of Theorem \ref{t1.1}]  First we observe that since $P(0)=1\not =P(1)=Q(1)+1$, $P(z)$ is critically injective polynomial. Also $P(z)-1$ and $P(z)-P(1)$ have a zero of multiplicity $n+1$ and $m+1$ respectively at $0$ and $1$, it follows that the zeros of $P(z)$ are simple.
Let the zeros be given by $\alpha _{j}$, $j=1,2,\ldots,n+m+1$.
Since $E_{f}(S,l)=E_{g}(S,l)$ it follows that $F$, $G$ share $(0,l)$.\\
{\bf Case 1}. If possible let us suppose that $H\not\equiv 0$. \\
{\bf Subcase 1.1.} $l\geq 1$.
While $l\geq 2$, using {\it Lemma \ref{l6}} we note that
\bea \label{e3.1}& & \ol N_{0}(r,0;g^{'})+\ol N(r,0;G\mid\geq 2)+ \ol N_{*}(r,0;F,G) \\&\leq& \ol N_{0}(r,0;g^{'})+\ol N(r,0;G \mid\geq 2)+\ol N(r,0;G\mid\geq 3)\nonumber\\ &\leq& N(r,0;g^{'}\vline g\not=0)+S(r,g)\nonumber\\&\leq&\ol N(r,0;g)+\ol N(r,\infty;g)+S(r,g).\nonumber \eea
Hence using (\ref{e3.1}), {\it Lemmas \ref{l1}}, {\it \ref{l2}} and {\it \ref{l3}} we get from second fundamental theorem for $\varepsilon >0$ that
\bea \label{e3.1a}& &\;\;(n+m+2)T(r,f)\\ &\leq & \ol N(r,\infty;f)+\ol N(r,0;f)+\ol N(r,1;f)+N(r,0;F\mid=1)+\ol N(r,0;F\mid\geq 2)\nonumber\\ & &-N_{0}(r,0;f^{'})+S(r,f)\nonumber\\&\leq & 2\ol N(r,0;f)+2\ol N(r,1;f)+\ol N(r,0;g)+\ol N(r,1;g)+ 2\ol N(r,\infty;f) + \ol N(r,\infty; g)\nonumber\\& & + \ol N_{0}(r,0;g') + \ol N(r,0;G\mid\geq 2) + \ol N_{*}(r,0;F,G) + S(r)\nonumber\\&\leq& 2\{\ol N(r,0;f)+\ol N(r,1;f)+ \ol N(r,\infty;f)+\ol N(r,0;g)+ \ol N(r,\infty; g)\}+\ol N(r,1;g) + S(r).\nonumber\\ &\leq & (11-2\Theta (0;f)-2\Theta (\infty;f)-2\Theta (1;f)-2\Theta (0;g)-2\Theta (\infty;g)-\Theta (1;g)+\varepsilon )T(r)+S(r).\nonumber\eea \par
In a similar way we can obtain
\bea \label{e3.1aa} (n+m+2)T(r,g)&\leq& (11-2\Theta (0;f)-2\Theta (\infty;f)-\Theta (1;f)-2\Theta (0;g)\\&&-2\Theta (\infty;g)-2\Theta (1;g)+\varepsilon )T(r)+S(r).\nonumber\eea
Combining (\ref{e3.1a}) and (\ref{e3.1aa}) we see that
\bea\label{e3.1aaa}&&(n+m-9+2\Theta (0;f)+2\Theta (\infty;f)+\Theta (1;f)+2\Theta (0;g)\\&&+2\Theta (\infty;g)+\Theta (1;g)+\min\{\Theta (1;f),\Theta (1;g)\}-\varepsilon )T(r) \leq S(r).\nonumber\eea
Since $\varepsilon > 0$ is arbitrary (\ref{e3.1aaa}) leads to a contradiction.\\
While $l=1$, using {\it Lemma \ref{l6}}, (\ref{e3.1}) can be changed to
\bea \label{e3.1e}& & \ol N_{0}(r,0;g^{'})+\ol N(r,0;G\mid\geq 2)+ \ol N_{*}(r,0;F,G) \\&\leq& \ol N_{0}(r,0;g^{'})+\ol N(r,0;G \mid\geq 2)+\ol N_{L}(r,0;G)+\ol N(r,0;F\mid\geq 3)\nonumber\\&\leq& N(r,0;g^{'}\vline g\not=0)+\sum\limits_{j=1}^{n+m+1}\ol N(r,\alpha _{j};f\mid\geq 3)\nonumber\\&\leq&\ol N(r,0;g)+\ol N(r,\infty;g)+\frac{1}{2}\sum \limits_{j=1}^{n+m+1}\{ N(r,\alpha_{j};f)-\ol N(r,\alpha_{j};f)\}+S(r,g)\nonumber\\&\leq& \ol N(r,0;g)+\ol N(r,\infty;g)+\frac{1}{2} N(r,0;f^{'}\mid f\not = 0)+S(r,g)\nonumber\\&\leq& \ol N(r,0;g)+\ol N(r,\infty;g)+\frac{1}{2}\{\ol N(r,0;f)+\ol N(r,\infty;f)\}+S(r,f)+S(r,g),\nonumber\eea
So using (\ref{e3.1e}), {\it Lemmas \ref{l2}} and {\it \ref{l3}} and proceeding as in (\ref{e3.1a}) we get from second fundamental theorem for $\varepsilon >0$ that
\bea \label{e3.1b}&&\;\; (n+m+2)T(r,f)\leq (12-2\Theta (0;f)-2\Theta (\infty;f)-2\Theta (1;f)\\&&-2\Theta (0;g)-2\Theta (\infty;g)-\Theta (1;g)+\varepsilon )T(r)+S(r)\nonumber.\nonumber\eea
Similarly we can obtain
\bea \label{e3.1bb} (n+m+2)T(r,g) &\leq& (12-2\Theta (0;f)-2\Theta (\infty;f)-\Theta (1;f)-2\Theta (0;g)\\&&-2\Theta (\infty;g)-2\Theta (1;g)+\varepsilon )T(r)+S(r).\nonumber\eea
Combining (\ref{e3.1b}) and (\ref{e3.1bb}) we see that
\bea\label{e3.1bbb}&&(n+m-10+2\Theta (0;f)+2\Theta (\infty;f)+\Theta (1;f)+2\Theta (0;g)\\&&+2\Theta (\infty;g)+\Theta (1;g)+\min\{\Theta (1;f),\Theta (1;g)\}-\varepsilon )T(r) \leq S(r).\nonumber\eea
Clearly, (\ref{e3.1bbb}) leads to a contradiction for $\varepsilon > 0$.\\
{\bf Subcase 1.2.} $l=0$.
Using {\it Lemma \ref{l6}} we note that
\bea \label{e3.1ee}& &\ol N_{0}(r,0;g^{'})+\ol N^{(2}_{E}(r,0;F)+2\ol N_{L}(r,0;G)+2\ol N_{L}(r,0;F)\\&\leq & \ol N_{0}(r,0;g^{'})+\ol N^{(2}_{E}(r,0;G)+\ol N_{L}(r,0;G)+\ol N_{L}(r,0;G)+2\ol N_{L}(r,0;F)\nonumber \\&\leq& \ol N_{0}(r,0;g^{'})+\ol N(r,0;G \mid\geq 2)+\ol N_{L}(r,0;G)+2\ol N_{L}(r,0;F)\nonumber\\&\leq& N(r,0;g^{'}\mid g\not=0)+\ol N(r,0;G \mid\geq 2)+2\ol N(r,0;F \mid\geq 2)\nonumber\\&\leq &\ol N(r,0;g)+\ol N(r,\infty;g)+\ol N(r,0;g)+\ol N(r,\infty;g)\nonumber\\ & &+2\ol N(r,0;f)+2\ol N(r,\infty;f)+S(r,f)+S(r,g).\nonumber\eea
Hence using (\ref{e3.1ee}), {\it Lemmas \ref{l1}}, {\it \ref{l2}} and {\it \ref{l3}} we get from second fundamental theorem for $\varepsilon >0$ that
\bea \label{e3.1c}\;\;\;\;& & (n+m+2)T(r,f)\\ &\leq & \ol N(r,\infty;f)+\ol N(r,0;f)+\ol N(r,1;f)+ N^{1)}_{E}(r,0;F)+\ol N_{L}(r,0;F)+\ol N_{L}(r,0;G)\nonumber\\& &+\ol N^{(2}_{E}(r,0;F)-N_{0}(r,0;f^{'})+S(r,f)\nonumber\\&\leq & 2\{\ol N(r,0;f)+\ol N(r,1;f)\}+\ol N(r,0;g)+\ol N(r,1;g)+ 2\ol N(r,\infty;f) + \ol N(r,\infty; g)\nonumber\\& &+\ol N^{(2}_{E}(r,0;F) +2\ol N_{L}(r,0;G)+2\ol N_{L}(r,0;F)+\ol N_{0}(r,0;g^{'})+S(r,f)+S(r,g)\nonumber \\ &\leq & (17-2\Theta (0;f)-2\Theta (\infty;f)-2\Theta (1;f)-2\Theta (0;g)\nonumber-2\Theta (\infty;g)\\&&-\Theta (1;g)+\varepsilon )T(r)+S(r)\nonumber.\eea
In a similar manner we can obtain \bea \label{e3.1cc} (n+m+2)T(r,g) &\leq& (17-2\Theta (0;f)-2\Theta (\infty;f)-\Theta (1;f)-2\Theta (0;g)\\&&-2\Theta (\infty;g)-2\Theta (1;g)+\varepsilon )T(r)+S(r).\nonumber\eea
Combining (\ref{e3.1c}) and (\ref{e3.1cc}) we see that
\bea \label{e3.1ccc}&&(n+m-15+2\Theta (0;f)+2\Theta (\infty;f)+\Theta (1;f)+2\Theta (0;g)\\&&+2\Theta (\infty;g)+\Theta (1;g)+\min\{\Theta (1;f),\Theta (1;g)\}-\varepsilon )T(r) \leq S(r).\nonumber\eea
Since $\varepsilon >0$, be arbitrary (\ref{e3.1ccc}) leads to a contradiction.\\
{\bf Case 2.} $H\equiv 0$. On integration we get from (\ref{e2.1}) \be\label{e3.4}\frac{1}{F}\equiv\frac{A}{G}+B,\ee where $A$, $B$ are constants and $A\not=0$.
From {\it Lemma \ref{l1}} we get \be\label{e3.5}T(r,f)=T(r,g)+S(r,g).\ee
{\bf Subcase 2.1.} First suppose that $B\not =0$.
From (\ref{e3.4}) we have $$\ol N(r,\infty;f)=\ol N(r,\frac{-A}{\;B};G).$$
{\bf Subcase 2.1.1.} Let $\frac{-A}{\;B}\not =1$.\\
If $\frac{-A}{\;B}\not =Q(1)+1$, then in view of (\ref{e3.5}), from the second fundamental theorem we get
\beas &&(n+2m+1)T(r,g)\\&\leq& \ol N(r,\infty;g)+\ol N\left(r,1;G\right)+\ol N\left(r,\frac{-A}{\;B};G\right)+S(r,g)\\&=&\ol N(r,\infty;g)+\ol N(r,0;g)+mT(r,g)+\ol N(r,\infty;f)+S(r,g)\\&\leq &(m+2)T(r,g)+\ol N(r,\infty;f)\}+S(r,g)\\&\leq& (m+3)T(r,g)+S(r,g),\eeas which  is a contradiction for $n \geq 3$ and $m\geq 3$.\\
\par
Next Suppose $\frac{-A}{\;B}=Q(1)+1$, from (\ref {e3.4}) we have
\bea\label {e3.5a}\frac{G}{BF}=G-P(1)=G+\frac{A}{B}=(g-1)^{m+1}(g-\alpha^{'} _{1})(g-\alpha^{'} _{2})\ldots (g-\alpha^{'} _{n}),\eea where $\alpha^{'} _{i}$, $i=1,2,\ldots,n$ are the distinct simple zeros of $P(z)+\frac{A}{B}$.
As $B\not= 0$, $f$ and $g$ do not have any common pole. Let $z_{0}$ be a zero of $g-1$ of multiplicity p(say) then it must be a pole of $f$ with multiplicity $q\geq 1$ (say). So from (\ref{e3.5a}) we have $$(m+1)p=(n+m+1)q\geq m+n+1.$$ i.e., $$p\geq \frac{n+m+1}{m+1}>1.$$
 Next suppose $z_{i}$ be a zero of $g-\alpha^{'} _{i}$ of multiplicity $p_{i}$, then in view of (\ref{e3.5a}), we have $z_{i}$ be a pole of $f$ of multiplicity $q_{i}$, (say) such that
$$p_{i}=(n+m+1)q_{i}\geq n+m+1.$$ Let $\beta _{j}$, $j=1,2,\ldots,m$ be the distinct simple zeros of $P(z)-1$. Now from the second fundamental theorem we get
\beas &&(n+m+1)T(r,g)\\&\leq& \ol N(r,\infty;g)+\ol N\left(r,1;G\right)+\ol N\left(r,\frac{-A}{\;B};G\right)+S(r,g)\\&\leq& \ol N(r,\infty;g)+\ol N(r,0;g)+\sum\limits_{j=1}^{m}\ol N(r,\beta _{j};g)+\ol N(r,1;g)+\sum\limits_{i=1}^{n}\ol N(r,\alpha^{'} _{i};g)+S(r,g)\\&\leq& \left(m+2+\frac{1}{2}+\frac{n}{n+m+1}\right)T(r,g)+S(r,g),\eeas which  is a contradiction for $n\geq 3$, $m\geq 3$.\\
{\bf Subcase 2.1.2.} Next let $\frac{-A}{\;B}=1$. From (\ref{e3.4}) we have $$\frac{1}{F}=\frac{B(G-1)}{G}.$$ Therefore in view of (\ref{e3.5}), second fundamental theorem yields
\beas&& T(r,g)+S(r,g)\geq \ol N(r,\infty;f)=\ol N(r,1;G)=\ol N(r,0;g)+\sum\limits_{j=1}^{m}\ol N(r,\beta _{j};g)\\&&\geq (m-1)T(r,g)+S(r,g),\eeas
a contradiction as $m\geq 3$.\\
{\bf Subcase 2.2.} $B=0$. From (\ref{e3.4}) we get \be\label{e3.6} AF\equiv G.\ee
{\bf Subcase 2.2.1.} Suppose $A\not =1$. \\
{\bf Subcase 2.2.1.1.} Let $A=P(1)$, then from (\ref{e3.6}) we have $$P(1)\left(F-\frac{1}{P(1)}\right)\equiv G-1.$$
As $P(1)\not =1$ and {\it Lemma \ref{l4}} implies $P(1)\not =-1$, we have $\frac{1}{P(1)}\not=P(1)$, it follows that $P(z)-\frac{1}{P(1)}$ has simple zeros. Let they be given by $\gamma _{i}$, $i=1,2,\ldots, n+m+1$.
%Then $$\ol N(r,1;g)+\sum\limits_{j=1}^{n}\ol N(r,\delta_{j};g)=\sum\limits_{i=1}^{n+m+1}\ol N(r,\gamma _{i};f).$$
So from the second fundamental theorem and (\ref{e3.5}) we get
\beas(n+m-1)T(r,f)&\leq& \sum\limits_{i=1}^{n+m+1}\ol N(r,\gamma _{i};f)+S(r,f)\\&\leq& \ol N(r,0;g)+\sum\limits_{j=1}^{m}\ol N(r,\beta _{j};g)\leq (m+1) T(r,f)+S(r,f),\eeas a contradiction since $n\geq 3$.\\
{\bf Subcase 2.2.1.2.} Let $A\not =P(1)$.\\
Then we have from (\ref{e3.6}) $$A(F-1)\equiv G-A.$$
%Clearly $0$ is a Picard exceptional value of $f$ since otherwise we have $G\equiv 0$, a contradiction.
Let the distinct zeros of $P(z)-A$ be given by $\delta _{i}$, $i=1,2,\ldots, n+m+1$.
 So from the second fundamental theorem and (\ref{e3.5}) we get
 \beas(n+m-1)T(r,g)&\leq& \sum\limits_{i=1}^{n+m+1}\ol N(r,\delta _{i};g)+S(r,g)\\&=& \sum\limits_{j=1}^{m}\ol N(r,\beta _{j};f)+\ol N(r,0;f)+S(r,f)\\&\leq& (m+1)T(r,g)+S(r,g),\eeas a contradiction since $n\geq 3$.\\
{\bf Subcase 2.2.2.} Suppose $A=1$. Then from (\ref{e3.6}) we have $F\equiv G$. i.e., $P(f)\equiv P(g)$. Here $k=2$, $d_{1}=0$, $d_{2}=1$, $q_{1}=n$, $q_{2}=m$. Since $\min \{q_{1},q_{2}\}=\min \{n,m \}\geq 2$ and $n+m\geq 5$
we see that $nm>n+m$. So from {\it Theorem B} we conclude that $P(z)$ is an uniqueness polynomial. Therefore $f\equiv g$. This proves the theorem. \end{proof}
\begin{center} {\bf Acknowledgement} \end{center}
The authors wish to thank the referee for his/her valuable remarks and suggestions to-wards the improvement of the paper.

\end{document}